\documentclass[12pt]{amsart}

\usepackage{microtype, mathrsfs, xfrac, amssymb, enumerate, xspace}
\usepackage{amsmath, amsthm, url}
\usepackage[latin1]{inputenc}
\usepackage[all]{xy}
\newtheorem{theorem}{Theorem}[section]
\newtheorem{lemma}[theorem]{Lemma}

\theoremstyle{definition}
\newtheorem{definition}[theorem]{Definition}
\newtheorem{example}[theorem]{Example}
\newtheorem{proposition}[theorem]{Proposition}

\newtheorem{corollary}[theorem]{Corollary}
\newtheorem{remark}[theorem]{Remark}

\newcommand{\e}{{\bf e}}
\newcommand{\bbA}{{\mathbb A}}

\DeclareMathOperator{\ed}{ed}
\DeclareMathOperator{\GL}{GL}
\DeclareMathOperator{\Mat}{Mat}

\DeclareMathOperator{\Et}{Et}
\DeclareMathOperator{\Sym}{S}

\DeclareMathOperator{\Gal}{Gal}

\DeclareMathOperator{\Aut}{Aut}
\DeclareMathOperator{\Lie}{Lie}
\DeclareMathOperator{\End}{End}

\DeclareMathOperator{\Hom}{Hom}
\DeclareMathOperator{\Spec}{Spec}
\DeclareMathOperator{\Alg}{Alg}
\DeclareMathOperator{\Char}{char}
\DeclareMathOperator{\trdeg}{trdeg}
\DeclareMathOperator{\Fields}{Fields}
\DeclareMathOperator{\Sets}{Sets}
\DeclareMathOperator{\Comm}{Comm}
\numberwithin{equation}{section}

\begin{document}

\title[Inseparable field extensions]{Essential dimension of inseparable field extensions}

\author{Zinovy Reichstein}
\address{Department of Mathematics\\
 University of British Columbia\\
 Vancouver, BC V6T 1Z2\\Canada}
 \email{reichst@math.ubc.ca}
\thanks{Partially supported by
 National Sciences and Engineering Research Council of
 Canada Discovery grant 253424-2017.}

\author{Abhishek Kumar Shukla}
\email{abhisheks@math.ubc.ca}
\thanks{Partially supported by a graduate fellowship from the Science and Engineering Research Board, India.}

\subjclass[2010]{Primary 12F05, 12F15, 12F20, 20G10}
	
%
%
\keywords{inseparable field extension, essential dimension, group scheme in prime characteristic}

\begin{abstract} Let $k$ be a base field, $K$ be a field containing $k$ and $L/K$ be a field extension of degree $n$. 
The essential dimension $\ed(L/K)$ over $k$ is a numerical invariant measuring ``the complexity" of $L/K$.
Of particular interest is
\[
\tau(n) = \max \{ \ed(L/K) \mid L/K \; \; \text{is a separable extension of degree $n$} \},
\]
also known as the essential dimension of the symmetric group $\Sym_n$.
The exact value of $\tau(n)$ is known only for $n \leqslant 7$.
In this paper we assume that $k$ is a field of characteristic $p > 0$ and study the essential dimension of 
inseparable extensions $L/K$. 
Here the degree $n = [L:K]$ is replaced by a pair $(n, \e)$ which accounts for the size
of the separable and the purely inseparable parts of $L/K$ respectively, and $\tau(n)$
is replaced by 
\[ \tau(n, \e) = \max \{ \ed(L/K) \mid L/K \; \; \text{is a field extension of type $(n, \e)$} \}.\]
The symmetric group $\Sym_n$ is replaced 
by a certain group scheme $G_{n,\e}$ over $k$. This group is neither finite nor smooth; nevertheless,
computing its essential dimension turns out to be easier than computing the essential 
dimension of $\Sym_n$. Our main result is a simple formula for $\tau(n, \e)$.
\end{abstract}

\maketitle

\section{Introduction}

Throughout this paper $k$ will denote a base field. All other fields will be assumed to contain $k$. 
A field extension $L/K$ of finite degree is said to descend to a subfield $K_0 \subset K$ if there extsts
a subfield $K_0 \subset L_0\subset L$ such that $L_0$ and $K$ generate $L$ and $[L_0:K_0]=[L:K]$. Equivalently, $L$ is isomorphic to
$L_0 \otimes_{K_0} K$ over $K$, as is shown in the following diagram.
\[ \xymatrix{    & L \ar@{-}[d]  \\
L_0 \ar@{-}[ur] \ar@{-}[d] & K   \\ 
K_0 \ar@{-}[ur]        &  } 
                 \]
\
	The essential dimension of $L/K$ (over $k$) is defined as
	\begin{equation*}
	     \ed(L/K)=\min\{ \trdeg(K_0/k) \mid L/K \; \; \text{descends to $K_0$ and $k \subset K_0$}  \}.
	\end{equation*}
	
Essential dimension of separable field extensions was studied in~\cite{buhler-reichstein}. Of particular interest is
\begin{equation} \label{e.tau}
\tau(n) = \max \{ \ed(L/K) \mid L/K \; \; \text{is a separable extension of degree $n$ and $k \subset K$} \}, 
\end{equation}
otherwise known as the essential dimension of the symmetric group $\Sym_n$. It is shown in~\cite{buhler-reichstein} that if $\operatorname{char}(k) = 0$, then
$\lfloor \dfrac{n}{2} \rfloor \leqslant \tau(n) \leqslant n-3$ for every $n \geqslant 5$.~\footnote{These inequalities hold for any base field $k$ of characteristic $\neq 2$.
On the other hand, the stronger lower bound of Theorem~\ref{thm.separable}, due to Duncan, is only known in characteristic $0$.}  
A.~Duncan~\cite{duncan} later strengthened the lower bound as follows.

\begin{theorem} \label{thm.separable} If $\Char(k) = 0$, then 
$\lfloor \dfrac{n + 1}{2} \rfloor \leqslant \tau(n) \leqslant n - 3$ for every $n \geqslant 6$. 
\end{theorem}

This paper is a sequel to~\cite{buhler-reichstein}. Here we will assume that $\Char(k) = p > 0$ and
study inseparable field extensions $L/K$. The role of the degree, $n = [L:K]$ in the separable case will be
played by a pair $(n, \e)$. The first component of this pair is the separable degree, $n = [S: K]$, where $S$ is the separable closure of $K$ in $L$. 
The second component is the so-called type
$\e = (e_1, \dots, e_r)$ of the purely inseparable extension $[L: S]$, where $e_1 \geqslant e_2 \geqslant \dots \geqslant e_r \geqslant 1$ are integers;
see Section~\ref{sect.type} for the definition.
Note that the type $\e = (e_1, \dots, e_r)$ uniquely determines the inseparable degree 
$[L:S] = p^{e_1+ \dots + e_r}$ of $L/K$ but not conversely. By analogy with~\eqref{e.tau} it is natural to define
\begin{equation} \label{e.tau'}
\tau(n, \e) = \max \{ \ed(L/K) \mid L/K \; \; \text{is a field extension of type $(n, \e)$ and $k \subset K$} \}. 
\end{equation}
Our main result is the following.

\begin{theorem} \label{thm.main} Let $k$ be a base field of characteristic $p > 0$, $n \geqslant 1$ and
$e_1 \geqslant e_2 \geqslant \dots \geqslant e_r \geqslant 1$ be integers,
$\e = (e_1, \dots, e_r)$ and $s_i = e_1 + \dots + e_i$ for $i = 1, \dots, r$. Then
\[ \tau(n, \e) =  n \sum_{i=1}^r  p^{s_i - i e_i} \, . \]
\end{theorem}

Some remarks are in order. 

\smallskip
(1) Theorem~\ref{thm.main} gives the exact value for $\tau(n, \e)$. This is in contrast to the 
separable case, where Theorem~\ref{thm.separable} only gives estimates and the exact value of $\tau(n)$ is unknown for any $n \geqslant 8$.

\smallskip
(2) A priori, the integers $\ed(L/K)$, $\tau(n)$ and $\tau(n, \e)$ all
depend on the base field $k$. However, Theorem~\ref{thm.main} shows that for a fixed $p = \Char(k)$,  
$\tau(n, \e)$ is independent of the choice of $k$.

\smallskip
(3) Theorem~\ref{thm.main} implies that for any inseparable extension $L/K$ of finite degree, 
\[ \ed(L/K) \leqslant \frac{1}{p} [L:K] \, ; \]
see Remark~\ref{rem.numerics}. This is again in contrast to the separable case, where Theorem~\ref{thm.separable} tells us that
there exists an extension $L/K$ of degree $n$ such that
$\ed(L/K) > \dfrac{1}{2} [L:K]$ for every odd $n \geqslant 7$ (assuming $\Char(k) = 0$).

\smallskip
(4) We will also show that the formula for $\tau(n, \e)$ remains valid if we replace essential dimension $\ed(L/K)$ in the
definition~\eqref{e.tau'} by essential dimension at $p$, $\ed_p(L/K)$; see Theorem~\ref{thm.main'}. For the definition of essential dimension at
a prime, see Section 5 in~\cite{reichstein-icm} or Section~\ref{sect.functor} below.

\smallskip
The number $\tau(n)$ has two natural interpretations. On the one hand, $\tau(n)$ is
the essential dimension of the functor $\Et_n$ which associates to a field
$K$ the set of isomorphism classes of \'etale algebras of degree $n$ over $K$.
On the other hand, $\tau(n)$ is the essential dimension of the symmetric group $\Sym_n$.
Recall that an \'etale algebra $L/K$ is a direct product $L = L_1 \times \dots \times L_m$ of separable field extensions $L_i/K$. Equivalently, an \'etale algebra 
of degree $n$ over $K$ can be thought of as a twisted $K$-form of the split algebra $k^n = k \times \dots \times k$ ($n$ times). 
The symmetric group $\Sym_n$ arises as the automorphism group of this split algebra, so that 
$\Et_n = H^1(K, \Sym_n)$; see Example~\ref{ex.etale}.

Our proof of Theorem~\ref{thm.main} relies on interpreting $\tau(n, \e)$ in a similar manner. 
Here the role of the split \'etale algebra $k^n$ will be
played by the  algebra $\Lambda_{n, \e}$, which is the direct product of $n$ copies of the truncated polynomial algebra 
\[ \Lambda_{\e} = k[x_1,\ldots,x_r]/(x_1^{p^e_1},\ldots,x_r^{p^{e_r}}). \]
Note that the $k$-algebra $\Lambda_{n, \e}$ is finite-dimensional, associative and commutative, but not semisimple.
\'Etale algebras over $K$ will get replaced by $K$-forms of $\Lambda_{n, \e}$. 
The role of the symmetric group $\Sym_n$ will be played
by the algebraic group scheme $G_{n,\e} = \Aut_k(\Lambda_{n, \e})$ over $k$. 
We will show that $\tau(n, \e)$ is the essential dimension
of $G_{n, \e}$, just like $\tau(n)$ is the essential dimension of $\Sym_n$ in the separable case.
The group scheme $G_{n, \e}$ is neither finite nor smooth; however, much to our surprise, 
computing its essential dimension turns out to be easier than computing the essential dimension of $\Sym_n$.

The remainder of this paper is structured as follows. Sections~\ref{sect.algebras} and \ref{sect.functor} contain preliminary results on finite-dimensional algebras, 
their automorphism groups and essential dimension. In Section~\ref{sect.type} we recall the structure theory of inseparable field extensions. Section~\ref{sect.versal} is devoted to versal algebras.
The upper bound of Theorem~\ref{thm.main} is proved in Section~\ref{sect.upper-bound}; alternative proofs are outlined in Section~\ref{sect.alternatives}. The lower bound of Theorem~\ref{thm.main}
is established in Section~\ref{sect.lower-bound}; our proof relies on 
the inequality~\eqref{e.tv} due to D.~Tossici~and A. Vistoli~\cite{tossici-vistoli}. 
Finally, in Section~\ref{sect.ee} we prove a stronger version of Theorem~\ref{thm.main} 
in the special case, where $n = 1$, $e_1 = \dots = e_r$, and $k$ is perfect.

\section{Finite-dimensional algebras and their automorphisms}
\label{sect.algebras}

Recall that in the Introduction we defined the essential dimension of a field extension $L/K$ of finite degree, where $K$ contains $k$.
The same definition is valid for any finite-dimensional algebra $A/K$. That is, we say that $A$ descends to a subfield $K_0$ if there exists
a $K_0$-algebra $A_0$ such that $A_0 \otimes_{K_0} K$ is isomorphic to $A$ (as a $K$-algebra). The essential dimension $\ed(A)$ is then
the minimal value of $\trdeg(K_0/k)$, where the minimum is taken over the
intermediate fields $k \subset K_0 \subset K$ such that $A$ descends to $K_0$.

Here by a $K$-algebra $A$ we mean a $K$-vector space with a bilinear 
``multiplication" map $m \colon A \times A \to A$. Later on we will primarily be interested in commutative associative
algebras with $1$, but at this stage $m$ can be arbitrary: we will not assume that $A$ is commutative, associative 
or has an identity element. (For example, one can talk of the essential dimension of a finite-dimensional Lie algebra $A/K$.) 
Recall that to each basis $x_1, \dots, x_n$ of $A$ one can associate a set of $n^3$ structure constants $c_{ij}^h \in K$, where
\begin{equation} \label{e.structure-constants}
x_i \cdot x_j = \sum_{h = 1}^n c_{ij}^h x_h \, . 
\end{equation}

\begin{lemma} \label{lem.structure-constants} Let $A$ be an $n$-dimensional $K$-algebra with structure constants $c_{ij}^h$
(relative to some $K$-basis of $A$). Suppose a subfield $K_0 \subset K$ contains $c_{ij}^h$ for every $i, j, h = 1, \dots, n$. 
Then $A$ descends to $K_0$.
In particular, $\ed(A) \leqslant \trdeg(K_0/k)$.
\end{lemma}

\begin{proof}
Let $A_0$ be the  $K_0$-vector space with basis $b_1, \dots, b_n$. Define the $K_0$-algebra structure on $A_0$ by~\eqref{e.structure-constants}.
Clearly $A_0 \otimes_{K_0} K = A$, and the lemma follows.
\end{proof}

The following lemma will be helpful to us in the sequel. 

\begin{lemma} \label{lem.algebra} 
Suppose $k \subset K \subset S$ are field extensions, such that $S/K$ is a separable of degree $n$.
Let $A$ be a finite-dimensional algebra over $S$. If $A$ descends to a subfield $S_0$ of $S$ such that $K(S_0) = S$, then
\[ \ed(A/K) \leqslant n \trdeg(S_0/k) \, . \]
Here $\ed(A/K)$ is the essential dimension of $A$, viewed as a $K$-algebra.
\end{lemma}

\begin{proof} By our assumption there exists an $S_0$-algebra $A_0$ such that $A = A_0 \otimes_{S_0} S$.

Denote the normal closure of $S$ over $K$ by $S^{\rm norm}$, and the associated Galois groups by $G = \Gal(S^{\rm norm}/K)$, $H = \Gal(S^{\rm norm}/S) \subset G$.
Now define $S_1 = k(g(s) \, | \, s \in S_0, \; \; g \in G)$. Choose a transcendence basis $t_1, \dots, t_d$ for $S_0$ over $k$,
where $d = \trdeg(S_0/k)$. Clearly $S_1$ is algebraic over $k(g(t_i) \, | \, g \in G, \; \; i = 1, \dots, d)$. Since $H$ fixes every element of $S$,
each $t_i$ has at most $[G:H] = n$ distinct 
translates of the form $g(t_i)$, $g \in G$. This shows that $\trdeg(S_1/k) \leqslant n d$. 

Now let $K_1 = S_1^{G} \subset K$ and $A_1 = A_0 \otimes_{K_0} K_1$. Since $S_1$ is algebraic over $K_1$, we have 
\[ \trdeg(K_1/k) = \trdeg(S_1/k) \leqslant n d . \]
Examining the diagram 
\[ \xymatrix{ A_0 \ar@{-}[d] \ar@{-}[r] & A_1 \ar@{-}[d] \ar@{-}[r]   & A \ar@{-}[d]  \\
              S_0  \ar@{-}[r]    & S_1 \ar@{-}[d] \ar@{-}[r]   & S \ar@{-}[d]  \\
                             & K_1 \ar@{-}[r]   & K,  } \]
we see that $A/K$ descends to $K_1$, and the lemma follows.                        
\end{proof}

Now let $\Lambda$ be a finite-dimensional $k$-algebra with multiplication map $m \colon \Lambda \times \Lambda \to \Lambda$.
The general linear group $\GL_k(\Lambda)$ acts on the vector space
$\Lambda^* \otimes_k \Lambda^* \otimes_k \Lambda$ of bilinear maps
$\Lambda \times \Lambda \to \Lambda$. The automorphsim group scheme $G = \Aut_k(\Lambda)$ of $\Lambda$
is defined as the stabilizer of $m$ under this action. It is a closed subgroup scheme of $\GL_k(\Lambda)$
defined over $k$. The reason we use the term ``group scheme" here, rather than ``algebraic group", is that
$G$ may not be smooth; see the Remark after Lemma III.1.1 in \cite{serre-gc}. 

%
%
\begin{proposition} \label{prop.aut-prod} Let $\Lambda$ be a commutative finite-dimensional local $k$-algebra with residue field $k$.
and $G = \Aut_k(\Lambda)$ be its automorphism group scheme.
Then the natural map \[ f \colon G^n \rtimes \Sym_n \to \Aut_k(\Lambda^n) \]
is an isomorphism. Here $G^n = G \times \dots \times G$ ($n$ times) acts on
$\Lambda^n = \Lambda \times \dots \times \Lambda$ ($n$ times) componentwise 
and $\Sym_n$ acts by permuting the factors. 
\end{proposition}

We begin with the following simple lemma.

\begin{lemma} \label{lem.idempotent} Let $\Lambda$ be a commutative finite-dimensional local $k$-algebra with residue field $k$
and $R$ be an arbitrary commutative $k$-algebra with $1$.
Then the only idempotents of $\Lambda_R = \Lambda \otimes_k R$ are those in $R$ (more precisely in $1\otimes R$).
\end{lemma} 

\begin{proof}  By Lemma~6.2 in~\cite{waterhouse}, the maximal ideal $M$ of $\Lambda$ consists of nilpotent elements.
Tensoring the natural projection $\Lambda \to \Lambda/M \simeq k$ with $R$, we obtain a surjective homomorphism
$\Lambda_R \to R$
whose kernel again consists of nilpotent elements. By Proposition 7.14 in~\cite{JacobsonII}, every idempotent in $R$
lifts to a unique idempotent in $\Lambda_R$, and the lemma follows.
\end{proof}

\begin{proof}[Proof of Proposition~\ref{prop.aut-prod}]
Let $\alpha_i=(0,\ldots,1,\ldots,0)$ where $1$ appears in the $i^{th}$ position. 
Then $\oplus_{i = 1}^n \, R\alpha_i$ is an $R$-subalgebra of $\Lambda^n_R$.
	
For any automorphism $f \in \Aut_R(\Lambda^n_R)$ consider the orthogonal idempotents $f(\alpha_1), \dots, f(\alpha_n)$. 
The components of each $f(\alpha_i)$ are idempotents in $\Lambda_R$. By Lemma~\ref{lem.idempotent}, they lie in $R$. 
Thus, $f(\alpha_i) \in \oplus_{i = 1}^n \,  R \alpha_i$. As a result, we obtain a morphism 
	\[\Aut_R(\Lambda_R^n) \xrightarrow{\; \tau_R \;} \Aut_R(\oplus_{i = 1}^n \,  R\alpha_i)=S_n(R).   \]
(For the second equality, see, e.g., p. 59 in \cite{waterhouse}.)
These maps are functorial in $R$ and thus give rise to a morphism $\tau \colon \Aut(\Lambda^n) \to \Sym_n$ of group schemes over $k$. 
The kernel of $\tau$ is $\Aut(\Lambda)^n$, and $\tau$ clearly has a section. The lemma follows.
\end{proof}

\begin{remark} \label{rem.con-commutative1} The assumption that $\Lambda$ is commutative in
Proposition~\ref{prop.aut-prod} can be dropped, as long as we assume that the center of $\Lambda$
is a finite-dimensional local $k$-algebra with residue field $k$. The proof proceeds along similar lines, except that we 
restrict $f$ to an automorphism of the center $Z(\Lambda^n) = Z(\Lambda)^n$ and apply Lemma~\ref{lem.idempotent} to 
$Z(\Lambda)$, rather than $\Lambda$ itself. 
This more general variant of Proposition~\ref{prop.aut-prod}
will not be needed in the sequel.
\end{remark}

 \begin{remark} \label{rem.con-commutative2} On the other hand, the residue field is $k$ cannot
be dropped. For example, if $\Lambda$ is a separable field extension of $k$ of degree $d$, then $\Aut_k(\Lambda^n)$ is a twisted form
of \[ \Aut_{\overline{k}} (\Lambda^n \otimes_k \overline{k}) = \Aut_{\overline{k}}(\overline{k}^{dn}) = \Sym_{nd}. \]
Here $\overline{k}$ denotes the separable closure of $k$. Similarly, $\Aut_k(\Lambda)^n \rtimes \Sym_d$ 
is a twisted form of $(\Sym_{d})^n \rtimes \Sym_n$.  For $d, n > 1$, these
groups have different orders, so they cannot be isomorphic.
\end{remark}

\section{Essential dimension of a functor}
\label{sect.functor}

In the sequel we will need the following general notion of essential dimension, due to A.~Merkurjev~\cite{berhuy-favi}.
Let $\mathcal{F} \colon \Fields_k \to \Sets$ be a covariant functor from the category of field extensions $K/k$ to the category of sets. Here $k$ is assumed to be fixed throughout,
and $K$ ranges over all fields containing $k$. We say that an object
$a\in \mathcal{F}(K)$ descends to a subfield $K_0 \subset K$ if $a$ lies in the image of the natural restriction map $\mathcal{F}(K_0)\to \mathcal{F}(K)$. The essential dimension $\ed(a)$ of $a$ is defined as minimal value of $\trdeg(K_0/k)$, where $k \subset K_0$ and $a$ descends to $K_0$.
The essential dimension of the functor $\mathcal{F}$, denoted by $\ed(\mathcal{F})$, is the supremum of $\ed(a)$ for all $a\in F(K)$, 
and all fields $K$ in $Fields_k$. 

If $l$ is a prime, there is also a related notion of essential dimension at $l$, which we denote by $\ed_l$. For an object $a \in \mathcal{F}$, we define
$\ed_l(a)$ as the minimal value of $\ed(a')$, where $a'$ is the image of $a$ in $\mathcal{F}(K')$, and the minimum is taken over all field extensions
$K'/K$ such that the degree $[K':K]$ is finite and prime to $l$. The essential dimension $\ed_l(\mathcal{F})$ of the functor $\mathcal{F}$ at $l$ is defined as the supremum of $\ed_l(a)$
for all $a\in F(K)$ and all fields $K$ in $Fields_k$. Note that the prime $l$ in this definition is unrelated to $p = \Char(k)$; we allow both $l = p$ and $l \neq p$.

\begin{example} \label{ex.group} Let $G$ be a group scheme over a base field $k$ and 
$\mathcal{F}_G \colon K \to H^1(K, G)$ be the functor defined by 
\[ \text{$\mathcal{F}_G(K) = \{$isomorphism classes of $G$-torsors $T \to \Spec(K) \}$.} \]
Here by a torsor we mean a torsor in the flat (fppf) topology. If $G$ is smooth, then $H^1(K, G)$ is the first Galois cohomology set, as in 
\cite{serre-gc}; see Section II.1.
The essential dimension $\ed(G)$ is, by definition, $\ed(\mathcal{F}_G)$, and similarly for the essential dimension $\ed_l(G)$ of $G$ at at prime $l$.
These numerical invariants of $G$ have been extensively studied; 
see, e.g.,\cite{merkurjev-chile} or~\cite{reichstein-icm} for a survey.
\end{example}

\begin{example} \label{ex.all-algebras} Define the functor $\Alg_n \colon K \to H^1(K, G)$ by 
\[ \text{$\Alg_n(K) = \{$isomorphism classes of $n$-dimensional $K$-algebras\}.} \]
If $A$ is an $n$-dimensional dimensional algebra, and $[A]$ is its class in $\Alg_n(K)$,
then $\ed([A])$ coincides with $\ed(A)$ defined at the beginning of Section~\ref{sect.algebras}.
By Lemma~\ref{lem.structure-constants} $\ed(\Alg_n) \leqslant n^3$;
the exact value is unknown (except for very small $n$). 
\end{example}

We will now restrict our attention to certain subfunctors of $\Alg_n$ which are better understood.

\begin{definition} \label{def.type}
Let $\Lambda/k$ be a finite-dimensional algebra and $K/k$ be a field extension (not necessarily finite or separable). We say that 
an algebra $A/K$ is a $K$-{\em form} of $\Lambda$ if
there exists a field $L$ containing $K$ such that $\Lambda \otimes_k L$ is isomorphic to $A \otimes_K L$ as an $L$-algebra.
We will write
\[ \Alg_{\Lambda} \colon \Fields_k \to \Sets \]
for the functor which sends a field $K/k$ to the set of $K$-isomorphism classes of $K$-forms of $\Lambda$. 
\end{definition}

\begin{proposition} \label{prop.auto}
Let $\Lambda$ be a finite-dimensional $k$-algebra and
$G = \Aut_k(\Lambda) \subset \GL(\Lambda)$ be its automorphism group scheme. Then the functors $\Alg_{\Lambda}$ and $\mathcal{F}_G = H^1(\ast, G)$ are isomorphic. 
In particular, $\ed(\Alg_\Lambda) = \ed(G)$ and $\ed_l(\Alg_{\Lambda}) = \ed_l(G)$ for every prime $l$.
\end{proposition}

\begin{proof} For the proof of the first assertion, see Proposition X.2.4 in~\cite{serre-lf} or Proposition III.2.2.2 in~\cite{knus}.
The second assertion is an immediate consequence of the first, since isomorphic functors have the same essential dimension.
\end{proof}

\medskip		
\begin{example} \label{ex.etale} The $K$-forms of $\Lambda_n = k \times \dots \times k$ ($n$ times) are called \'etale algebras of degree $n$. An \'etale algebra $L/K$ of degree $n$
is a direct products of separable field extensions,
\[ \text{$L = L_1 \times \dots \times L_r$, where $\sum_{i = 1}^r [L_i:K] = n$.} \]
The functor $\Alg_{\Lambda_n}$ is usually denoted by $\Et_n$. The automorphism group $\Aut_k(\Lambda_n)$ is the symmetric group $\Sym_n$, acting
on $\Lambda_n$ by permuting the $n$ factors of $k$; see Proposition~\ref{prop.aut-prod}.
Thus
$\Et_n = H^1(K, \Sym_n)$; see, e.g., Examples 2.1 and 3.2 in~\cite{serre-ci}.
\end{example}

\section{Field extensions of type $(n, \e)$}
\label{sect.type}

Let $L/S$ be a purely inseparable extension of finite degree. 
For $x\in L$ we define the exponent of $x$ over $S$ as the smallest integer $e$ such that $x^{p^e}\in S$. 
We will denote this number by $e(x,S)$. We will say that $x\in L$ is {\em normal} in $L/S$
if $e(x) = \max\{e(y)\mid y\in L \}$. A sequence $x_1,\ldots,x_r$ in $L$ is called normal 
if each $x_i$ is normal in $L_i/L_{i-1}$ and $x_i\notin L_{i-1}$. Here $L_i=S(x_1,\ldots,x_{i-1})$ 
and $L_0=S$. If $L=S(x_1,\ldots,x_r)$, where $x_1,\ldots,x_r$ is a normal sequence in $L/S$, 
then we call $x_1,\ldots,x_r$ a {\em normal generating sequence} of $L/S$. We will say that
this sequence is {\em of type} $\e = (e_1, \dots, e_r)$ if $e_i:=e(x_i, L_{i-1})$ for each $i$. 
Here $L_i = S(x_1, \dots, x_i)$, as above. It is clear that $e_1\geqslant e_2 \geqslant \ldots \geqslant e_r$. 

\begin{proposition} {\rm (G.~Pickert}~\cite{pickert}{\rm)} \label{prop.type}
Let $L/S$ be a purely inseparable field extension of finite degree. 

\smallskip
	(a) For any generating set $\Lambda$ of $L/S$ there exists a normal generating sequence $x_1, \dots, x_r$ with each $x_i \in \Lambda$. 
	
\smallskip
	(b) If $x_1, \dots, x_r$ and $y_1, \dots, y_s$ are two normal generating sequences for $L/S$, of
	types $(e_1, \dots, e_r)$ and $(f_1, \dots, f_s)$ respectively, then $r = s$ and $e_i = f_i$ for each $i = 1, \dots, r$. 
\end{proposition}

\begin{proof} For modern proofs of both parts, see Propositions 6 and 8 in \cite{rasala} or Lemma 1.2 and Corollary 1.5 in \cite{karpilovsky}.
\end{proof}

Proposition~\ref{prop.type} allows us to talk about {\em the type} of a purely inseparable extension $L/S$. We say that $L/S$ is of type $\e = (e_1, \dots, e_r)$ if 
it admits a normal generating sequence $x_1, \dots, x_r$ of type $\e$.

Now suppose $L/K$ is an arbitrary inseparable (but not necessarily purely inseparable) field extension $L/K$ of finite degree. 
Denote the separable closure of $K$ in $L$ by $S$. We will say that $L/K$ is of type $(n, \e)$ if $[S:K] = n$ and the purely inseparable
extension $L/S$ is of type $\e$. 

\begin{remark} \label{rem.finite-field}
Note that we will assume throughout that $r \geqslant 1$, i.e., that $L/K$ is not separable. In particular, a finite field $K$ does not admit an extension of type $(n, \e)$
for any $n$ and $\e$.
\end{remark}

\begin{remark}\label{rem.becker-maclane}
It is easy to see that any proper subset of a normal generating sequence $\{x_1,\ldots,x_r\}$ of purely inseparable extension $L/K$ generates a proper subfield of $L$. In other words,
a normal generating sequence is a minimal generating set of $L/K$. By Theorem 6 in \cite{becker-maclane} we have $[L: K(L^p)]=p^r$. Here $K(L^p)$ denotes the subfield of $L$ generated by $L^p$ and $K$.
\end{remark}

\begin{lemma} \label{lem.existence} Let $n \geqslant 1$ and 
$e_1 \geqslant e_2 \geqslant \dots \geqslant e_r \geqslant 1$ be integers. Then there exist

\smallskip
(a) a separable field extension $E/F$ of degree $n$ with $k \subset F$,

(b) a field extension $L/K$ of type $(n, \e)$ with $k \subset K$ and $\e = (e_1, \dots, e_r)$.
\end{lemma}

In particular, this lemma shows that the maxima in definitions~\eqref{e.tau} and~\eqref{e.tau'} are taken over a non-empty set of integers. 

\begin{proof}  (a) Let $x_1, \dots, x_n$ be independent variables over $k$.
Set $E = k(x_1, \dots, x_n)$ and $F = E^C$, where $C$ is the cyclic group of order $n$ acting on $E$ by permuting the variables.
Clearly $E/F$ is a Galois (and hence, separable) extension of degree $n$.

(b) Let $E/F$ be as in part (a) and $y_1, \dots, y_r$ be independent variables over $F$. Set $L = E(y_1, \dots, y_r)$ and $K = F(z_1, \dots, z_r)$,
where $z_i = y_i^{p^{e_i}}$. 
One readily checks that $S = E(z_1, \dots, z_n)$ is the separable closure of $K$ in $L$ and $L/S$ is a purely inseparable extension of type $\e$.
\end{proof}

Now suppose $n \geqslant 1$ and $\e = (e_1, \dots, e_r)$ are as above, with $e_1 \geqslant e_2 \geqslant \dots \geqslant e_r \geqslant 1$.
The following finite-dimensional commutative $k$-algebras will play an important role in the sequel:
\begin{equation} \label{e.Lambda} \text{$\Lambda_{n, \e} = \Lambda_{\e} \times \dots \times \Lambda_{\e}$ ($n$ times), where 
$\Lambda_{\e} = k[x_1, \ldots, x_r]/(x_1^{p^{e_1}},\ldots, x_r^{p^{e_r}})$} 
\end{equation}
is a truncated polynomial algebra.  

\begin{lemma} \label{lem.uniqueness} 
$\Lambda_{n, \e}$ is isomorphic to $\Lambda_{m, {\bf f}}$ if and only if $m = n$ and $\e = {\bf f}$.
\end{lemma}

\begin{proof} Note that $\Lambda_{\e}$ is a finite-dimensional local $k$-algebra with residue field $k$. By Lemma~\ref{lem.idempotent},
the only idempotents in $\Lambda_{\e}$ are $0$ and $1$. This readily implies that the only idempotents in $\Lambda_{n, \e}$ are
of the form $(\epsilon_1, \dots, \epsilon_n)$, where each $\epsilon_i$ is $0$ or $1$, and the only minimal idempotents are
\[ \text{$\alpha_1 = (1, 0, \dots, 0)$, $\ldots$ , $\alpha_n = (0, \dots, 0, 1)$.} \]
(Recall that a minimal idempotent is one that cannot be written as a product of two orthogonal idempotents.)
Suppose $\Lambda_{n, \e}$ and $\Lambda_{m, {\bf f}}$ are isomorphic.
Then they have the same number of minimal idempotents; hence, $m = n$. Denote the minimal idempotents of $\Lambda_{m, {\bf f}}$
by
\[ \text{$\beta_1 = (1, 0, \dots, 0)$, $\ldots$ , $\beta_m = (0, \dots, 0, 1)$.} \]
A $k$-algebra isomorphism  
$\Lambda_{n, \e} \to \Lambda_{m, {\bf f}}$ takes $\alpha_1$ to $\beta_j$ for some $j = 1, \dots, n$
and, hence, induces a $k$-algebra isomorphism between $\alpha_1 \Lambda_{n, \e} \simeq \Lambda_{\e}$ and 
$\beta_j \Lambda_{m, {\bf f}} \simeq \Lambda_{\bf f}$. To complete the proof, we appeal to Proposition 8 in~\cite{rasala},
which asserts that $\Lambda_{\bf e}$ and $\Lambda_{\bf f}$ are isomorphic if and only if $\e = {\bf f}$.
\end{proof}

\begin{lemma} \label{lem.type} Let $L/K$ be a field extension of finite degree. Then the following are equivalent.

\smallskip
(a) $L/K$ is of type $(n, \e)$. 

\smallskip
(b) $L$ is a $K$-form of $\Lambda_{n, \e}$. In other words,
$L \otimes_K K'$ is isomorphic to $\Lambda_{n, \e} \otimes_k K'$ as an $K'$-algebra for some field extension $K'/K$.
\end{lemma}

\begin{proof} (a)  $\Longrightarrow$ (b): Assume $L/K$ is a field extension of type $(n, \e)$.
Let $S$ be the separable closure of $K$ in $L$ and $K'$ be an algebraic closure of $S$
(which is also an algebraic closure of $K$). Then 
\[ L \otimes_K K' = L \otimes_S (S \otimes_K K') = (L \otimes_{S} K') \times \dots \times (L \otimes_{S} K') \; \text{($n$ times).} \]
On the other hand, by \cite{rasala}, Theorem 3,
$L \otimes_S K'$ is isomorphic to $\Lambda_{\e}$ as a $K'$-algebra, and part (b) follows.

\smallskip
(b) $\Longrightarrow$ (a): Assume $L \otimes_K K'$ is isomorphic to $\Lambda_{n, \e} \otimes_k K'$ as an $K'$-algebra for some field extension $K'/K$.
After replacing $K'$ by a larger field, we may assume that $K'$ contains the normal closure of $S$ over $K$.
Since $\Lambda_{n, \e} \otimes_k K'$ is not separable over $K'$, $L$ is not separable over $K$.
Thus $L/K$ is of type $(m, {\bf f})$ for some $m \geqslant 1$ and ${\bf f} = (f_1, \dots, f_s)$ with
$f_1 \geqslant f_2 \geqslant \dots \geqslant f_s \geqslant 1$. By part (a), $L \otimes_K K''$ is isomorphic to $\Lambda_{m, {\bf f}} \otimes_k K''$
for a suitable field extension $K''/K$. After enlarging $K''$, we may assume without loss of generality that $K' \subset K''$. We conclude that
$\Lambda_{n, \e} \otimes_k K''$ is isomorphic to $\Lambda_{m, {\bf f}} \otimes_k K''$ as a $K''$-algebra. 
By Lemma~\ref{lem.uniqueness}, with $k$ replaced by $K''$, this is only possible if $(n, \e) = (m, {\bf f})$. 
\end{proof}

\section{Proof of the upper bound of Theorem~\ref{thm.main}}
\label{sect.upper-bound}

In this section we will prove the following proposition.

\begin{proposition} \label{prop.upper-bound}
Let $n \geqslant 1$ and $\e = (e_1, \dots, e_r)$, where $e_1 \geqslant \dots \geqslant e_r \geqslant 1$. 
Then
\[ \tau(n, \e) \leqslant  n  \sum_{i=1}^r  p^{s_i - i e_i}  \, . \]
\end{proposition}

Our proof of Proposition~\ref{prop.upper-bound} will be facilitated by the following lemma.

\begin{lemma} \label{lem.frobenius} Let $K$ be an infinite field of characteristic $p$, $q$ be a power of $p$,
$S/K$ be a separable field extension of finite degree,
and $0 \neq a \in S$. Then there exists an $s \in S$ such that $as^q$ is a primitive element for $S/K$. 
\end{lemma}

\begin{proof} Assume the contrary. It is well known that there are only finitely many intermediate fields between $K$ and $S$; see~e.g.,~\cite{lang}, Theorem V.4.6.
Denote the intermediate fields properly contained in $S$ by $S_1, \dots, S_n \subsetneq S$ and let $\bbA_K(S)$ be the affine space associated to $S$.
(Here we view $S$ as a $K$-vector 
space.) The non-generators of $S/K$ may now be viewed as $K$-points of the finite union 
\[ Z = \cup_{i= 1}^n \, \bbA_K(S_i) \, . \]
Since we are assuming that every element of $S$ of the form $as^q$ is a non-generator, and $K$ is an infinite field, 
the image of the $K$-morphism $f \colon \bbA(S) \to \bbA(S)$ given by $s \to as^q$ 
lies in $Z = \cup_{i = 1}^n \, \bbA_K(S_i)$. Since $\bbA_K(S)$ is irreducible, 
we conclude that the image of $f$ lies in one of the affine subspaces $\bbA_K(S_i)$, say in
$\bbA_K(S_1)$. Equivalently, $as^q \in S_1$ for every $s \in S$. Setting $s = 1$, we see that $a \in S_1$. Dividing $as^q \in S_1$ by
$0 \neq a \in S_1$, we conclude that
$s^q \in S_1$ for every $s \in S$. Thus $S$ is purely inseparable over $S_1$, contradicting our assumption that $S/K$ is separable.
\end{proof}

\begin{proof}[Proof of Proposition~\ref{prop.upper-bound}]
Let $L/K$ be a field extension of type $(n, \e)$. 
Our goal is to show that $\ed(L/K) \leqslant  n \sum_{j=1}^r  p^{s_j - je_j}$.
By Remark~\ref{rem.finite-field}, $K$ is infinite. 

Let $S$ be the separable closure of $K$ in $L$ and $x_1, \dots, x_r$ be a normal generating sequence for 
the purely inseparable extension $L/S$ of type $\e$. Set $q_i = p^{e_i}$. Recall that by the definition
of normal sequence, $x_1^{q_1} \in S$. We are free to replace $x_1$ 
by $x_1 s$ for any $0 \neq s \in S$; clearly $x_1s, x_2, \dots, x_r$ is another normal generating sequence.
By Lemma~\ref{lem.frobenius}, we may choose $s \in S$ so that  $(x_1 s)^{q_1}$ is a primitive element for $S/K$.
In other words, we may assume without loss of generality that $z = x_1^{q_1}$ is a primitive element for 
$S/K$.

By the structure theorem of Pickert, each $x_i^{q_i}$ lies in $S[x_1^{q_i}, \dots, x_{i-1}^{q_i}]$, where $q_i = p^{e_i}$;
see Theorem 1 in~\cite{rasala}. In other words, for each $i = 1, \dots, r$, 
\begin{equation} \label{e.structure} x_i^{q_i} = \sum a_{d_{1}, \dots, d_{i-1}}  x_1^{q_i d_1} \dots x_{i-1}^{q_i d_{i-1}}
\end{equation}
for some for some $a_{d_1, \dots, d_{i-1}} \in S$. Here the sum is taken over all integers $d_1, \dots, d_{i-1}$ 
between $0$ and $p^{e_j - e_i} - 1$. 
By Lemma~\ref{lem.structure-constants}, $L$ (viewed as an $S$-algebra), descends to
\[ S_0 = k(a_{d_1, \dots, d_{i-1}} \, | \, \text{$i = 1, \dots, r$ and $ 0 \leqslant d_j \leqslant p^{e_j - e_i} - 1$} ) \, . \]
Note that for each $i = 1, \dots, r$, there are exactly 
\[ p^{e_1 - e_i} \cdot p^{e_2 - e_i} \cdot \ldots \cdot p^{e_{i-1} -e_i} = p^{s_i - i e_i} \]
choices of the subscripts $d_1, \ldots, d_{i-1}$. Hence, $S_0$ is generated over $k$ by $\sum_{i=1}^r  p^{s_i - i e_i}$ elements
and consequently, 
\[ \trdeg(S_0/k) \leqslant \sum_{i=1}^r  p^{s_i - i e_i}. \]
Applying Lemma~\ref{lem.algebra} with $L = A$, we see that
$\ed(L/K) \leqslant n \trdeg(S_0/k)$, and the proposition follows.
\end{proof}

\begin{remark} \label{rem.numerics} Suppose $L/K$ is an extension of type $(n, \e)$, where $e = (e_1, \dots, e_r)$.
Here, as usual, $K$ is assumed to contain the base field $k$ of characteristic $p > 0$. Dividing both sides of the inequality in Proposition~\ref{prop.upper-bound}
by $[L:K] = n p^{e_1 + \dots + e_r}$, we readily deduce that
\[ \frac{\ed(L/K)}{[L:K]} \leqslant \frac{r}{p^r} \leqslant \frac{1}{p} \, . \]
In particular, $\ed(L/K) \leqslant \dfrac{1}{2}[L:K]$ for any inseparable extension $[L:K]$ of finite degree, in any (positive) characteristic.
As we pointed out in the Introduction, this inequality fails in characteristic $0$ (even for $k = \mathbb C$).
\end{remark} 

\section{Versal algebras}
\label{sect.versal}

Let $K$ be a field and $A$ be a finite-dimensional associative $K$-algebra with $1$. Every $a \in A$ gives rise to the $K$-linear map $l_a \colon A \to A$ given by
$l_a(x) = ax$ (left multiplication by $a$). Note that $l_{ab} = l_a \cdot l_b$. It readily follows from this that $a$ has a multiplicative
inverse in $A$ if and only if $l_a$ is non-singular.

\begin{proposition} \label{prop.versal} 
Let $l$ be a prime integer and $\Lambda$ be a finite-dimensional associative $k$-algebra with $1$. 
Assume that there exists a field extension $K/k$ and a $K$-form
$A$ of $\Lambda$ such that $A$ is a division algebra. Then 

\smallskip
(a) there exists a field $K_{ver}$ containing $k$ and a $K_{ver}$-form $A_{ver}$ of $\Lambda$ such that
$\ed(A_{ver}) = \ed(\Alg_{\Lambda})$, $\ed_l(A_{ver}) = \ed_l(\Alg_{\Lambda})$, and $A_{ver}$ is a division algebra. 

\smallskip
(b) If $G$ is the automorphism group scheme of $\Lambda$, then 
\[ \ed(G) = \ed(\Alg_{\Lambda}) = \max \{ \ed(A/K) \, | \, \text{$A$ is a $K$-form of $\Lambda$ and a division algebra} \} \]
and 
\[ \ed_l(G) = \ed_l(\Alg_{\Lambda}) = \max \{ \ed_l(A/K) \, | \, \text{$A$ is a $K$-form of $\Lambda$ and a division algebra} \}. \]
\end{proposition}

Here the subscript ``ver" is meant to indicate that $A_{ver}/K_{ver}$ is a versal object for $\Alg_{\Lambda} = H^1(\ast, G)$.
For a discussion of versal torsors, see Section I.5 in~\cite{serre-ci} or \cite{duncan-reichstein}.

\begin{proof} (a) We begin by constructing of a versal $G$-torsor $T_{ver} \to \Spec(K_{ver})$.
Recall that $G = \Aut_k(\Lambda)$ is defined as a closed subgroup of the general linear 
group $\GL_k(\Lambda)$.
This general linear group admits a generically free linear action on some vector space $V$ (e.g., we can take $V = \End_k(\Lambda)$, with the natural left $G$-action). 
Restricting to $G$ we obtain a generically free representation $G \to \GL(V)$. We can now choose a dense open $G$-invariant 
subscheme $U \subset V$ over $k$ which is the total space of 
a $G$-torsor $\pi \colon U \to B$; see, e.g., Example 5.4 in~\cite{serre-ci}.
Passing to the generic point of $B$, we obtain a $G$-torsor $T_{ver} \to \Spec(K_{ver})$, where $K_{ver}$ is the function field of $B$ over $k$. 
Then $\ed(T_{ver}/K_{ver}) = \ed(G)$ (see, e.g., Section 4 in \cite{berhuy-favi}) and 
$\ed_l(T_{ver}/K_{ver}) = \ed_l(G)$ (see Lemma 6.6 in~\cite{reichstein-youssin} or Theorem 4.1 in~\cite{merkurjev-chile}).

Let $T \to \Spec(K)$ be the torsor associated to the $K$-algebra $A$
and $A_{ver}$ be the $K_{ver}$-algebra associated to $T_{ver} \to \Spec(K_{ver})$
under the isomorphism between the functors $\Alg_{\Lambda}$ and $H^1( \ast, G)$ of Proposition~\ref{prop.auto}.
By the characteristic-free version of the no-name Lemma, proved in \cite{reichstein-vistoli}, Section 2, 
$T \times V$ is $G$-equivariantly birationally isomorphic to $T \times \bbA_k^d$, where $d = \dim(V)$ and
$G$ acts trivially on $\bbA_k^d$. In other words,
we have a Cartesian diagram of rational maps defined over $k$
\[ \xymatrix{  & T \times \bbA^d \ar@{-->}[r]^{\simeq} \ar@{->}[d] & T \times V \ar@{-->}[r]^{\; \text{pr}_2} &  U \ar@{->}[d] \\ 
                 \bbA_K^d \ar@{=}[r]        &  \Spec(K) \times \bbA^d \ar@{-->}[rr] &  & B.}
\]
Here all direct products are over $\Spec(k)$, and
$\text{pr}_2$ denotes the rational $G$-equivariant projection 
map taking $(t, v) \in T \times V$ to $v \in V$ for $v \in U$.
The map $\Spec(K) \times \bbA^d \dasharrow B$ in the bottom row is induced from the dominant $G$-equivariant map $T \times \bbA^d \dasharrow U$ on top. 
Passing to generic points, we obtain an inclusion of field $K_{ver} \hookrightarrow K(x_1, \dots, x_d)$ such 
that the induced map $H^1(K_v, G) \to H^1(K(x_1, \dots, x_d), G)$ sends the class of
$T_{ver} \to \Spec(K_{ver})$ to the class associated to $T \times \bbA^d \to \bbA_K^d$. Under the isomorphism of Proposition~\ref{prop.auto}
between the functors $\Alg_{\Lambda}$ and $\mathcal{F}_G = H^1(\ast, G)$, this translates to 
\[ A_{ver} \otimes_{K_{ver}} K(x_1, \dots, x_d) \simeq A \otimes_K K(x_1, \dots, x_d) \]
as $K(x_1, \dots, x_d)$-algebras. 

For simplicity we will write $A(x_1, \dots, x_d)$ in place of $A \otimes_K K(x_1, \dots, x_d)$.
Since $A$ is a division algebra, so is $A(x_1, \dots, x_d)$. Thus the linear map $l_a \colon A(x_1, \dots, x_d) \to A(x_1, \dots, x_d)$ 
is non-singular (i.e., has trivial kernel) for every $a \in A_{ver}$.
Hence, the same is true for the restriction of $l_a$ to $A_{ver}$. We conclude that $A_{ver}$ is a division algebra.
Remembering that $A_{ver}$ corresponds to $T_{ver}$ under
the isomorphism of functors between $\Alg_{\Lambda}$ and $\mathcal{F}_{G}$, we see that
\[ \ed(A_{ver}) = \ed(T_{ver}/K_{ver}) = \ed(G) = \ed(\Alg_{\Lambda}) \]
and
\[ \ed_l(A_{ver}) = \ed_l(T_{ver}/K_{ver}) = \ed_l(G) = \ed_l(\Alg_{\Lambda}) \, , \]
as desired.

\smallskip
(b) The first equality in both formulas follows from Proposition~\ref{prop.auto}, and the second from part (a).
\end{proof}

We will now revisit the finite-dimensional $k$-algebras $\Lambda_{\e}$ and $\Lambda_{n, \e} = 
\Lambda_{\e} \times \dots \times \Lambda_{\e}$ ($n$ times) defined in Section~\ref{sect.type}; see~\eqref{e.Lambda}.
We will write $G_{n, \e} = \Aut(\Lambda_{n, \e}) \subset \GL_k(\Lambda_{n, \e})$ for the automorphism group scheme of $\Lambda_{n, \e}$ and
$\Alg_{n, \e}$ for the functor $\Alg_{\Lambda_{n, \e}} \colon \Fields_k \to \Sets$. Recall that this functor 
associates to a field $K/k$ the set of isomorphism classes of $K$-forms of $\Lambda_{n, \e}$.

Replacing essential dimension by essential dimension by essential dimension at a prime $l$  in the definitions~\eqref{e.tau} and
\eqref{e.tau'} or $\tau(n)$ and $\tau(n, \e)$ respectively, we define
\[
\tau_l(n) = \max \{ \ed_l(L/K) \mid L/K \; \; \text{is a separable field extension of degree $n$ and $k \subset K$} \}. 
\]
and
\[
\tau_l(n, \e) = \max \{ \ed_l(L/K) \mid L/K \; \; \text{is a field extension of type $(n, \e)$ and $k \subset K$} \}. 
\]

\begin{corollary} \label{cor.tau} Let $l$ be a prime integer. Then

(a) $\ed(\Sym_n) = \ed(\Et_n) = \tau(n)$ and $\ed_l(\Sym_n) = \ed_l(\Et_n) = \tau_l(n)$.
Here $\Et_n$ is the functor of $n$-dimensional \'etale algebras, as in Example~\ref{ex.etale}.

\smallskip
(b) $\ed(G_{n, \e}) = \ed(\Alg_{n, \e}) = \tau(n, \e)$ and $\ed_l(G_{n, \e}) = \ed_l(\Alg_{n, \e}) = \tau_l(n, \e)$.
\end{corollary}

\begin{proof}  (a) Recall that \'etale algebra are, by definition, 
commutative and associative with identity. For such algebras ``division algebra" is the same as ``field".
By Lemma~\ref{lem.existence}(a) there exists a separable field extension $E/F$ of degree $n$ with $k \subset F$.
The desired equality follows from Proposition~\ref{prop.versal}(b).

(b) The same argument as in part (a) goes through, with part (a) of Lemma~\ref{lem.existence} replaced by part (b). 
\end{proof}

\begin{remark} The value of $\ed_l(\Sym_n)$ is known:
\[ \ed_l(\Sym_n) =  \begin{cases} \text{$\lfloor \dfrac{n}{l} \rfloor$, if $\Char(k) \neq l$, see Corollary 4.2 in \cite{meyer-reichstein},} \\
                                   \text{$1$, if $\Char(k) = l \leqslant n$, see Theorem 1 in \cite{reichstein-vistoli-p}, and} \\
\text{$0$, if $\Char(k) = l > n$, see Lemma 4.1 in~\cite{meyer-reichstein} or Theorem 1 in \cite{reichstein-vistoli-p}.}
\end{cases}
\]
\end{remark}

 \section{Conclusion of the proof of Theorem~\ref{thm.main} } 
 \label{sect.lower-bound}
 
 In this section we will prove Theorem~\ref{thm.main} in the following strengthened form. 

\begin{theorem} \label{thm.main'} Let $k$ be a base field of characteristic $p > 0$, $n \geqslant 1$ and
$e_1 \geqslant e_2 \geqslant \dots \geqslant e_r \geqslant 1$ be integers,
$\e = (e_1, \dots, e_r)$ and $s_i = e_1 + \dots + e_i$ for $i = 1, \dots, r$. Then
\[ \tau_p(n, \e) = \tau(n, \e) =  n \sum_{i=1}^r  p^{s_i - i e_i} \, . \]
\end{theorem}

By definition $\tau_p(n, \e) \leqslant \tau(n, \e)$ and by
Proposition~\ref{prop.upper-bound},  $\tau(n, \e) \leqslant  n \sum_{i=1}^r  p^{s_i - i e_i}$.
 Moreover, by Corollary~\ref{cor.tau}(b), $\tau_p(n, \e) = \ed_p(G_{n, \e})$.
It thus remains to show that
 \begin{equation} \label{e.lower-bound}
 \ed_p(G_{n, \e}) \geqslant n \sum_{i=1}^r  p^{s_i - i e_i} \, . 
 \end{equation}
 Our proof of~\eqref{e.lower-bound} will be based on the following general inequality, due to Tossici and Vistoli~\cite{tossici-vistoli}:
 \begin{equation} \label{e.tv} \ed_p(G)\geqslant \dim(\Lie(G))-\dim(G) \end{equation}
 for any group scheme $G$ of finite type over a field $k$ of characteristic $p$.  Now recall that $G_{\e} = \Aut_k(\Lambda_{\e})$, and
 $G_{n, \e} = \Aut_k(\Lambda_{n, \e})$, where $\Lambda_{n, \e} = \Lambda_{\e}^n$. Since
 $\Lambda_{\e}$ is is a commutative local $k$-algebra with residue field $k$, 
 Proposition~\ref{prop.aut-prod} tells us that $G_{n, \e} = G_{\e}^n \rtimes \Sym_n$ (see also Proposition 5.1 in \cite{salas}).
 We conclude that
 \[ \text{$\dim(G_{n, \e}) = n \dim(G_{\e})$ and $\dim(\Lie(G_{n, \e})) = n \dim(\Lie(G_{\e}))$.} \]
 Substituting these formulas into~\eqref{e.tv}, we see that the proof of the inequality~\eqref{e.lower-bound} (and thus of
 Theorem~\ref{thm.main'}) reduces to the following.
  
 \begin{proposition} \label{prop.dimensions} Let $\e = (e_1, \dots, e_r)$, where $e_1 \geqslant \dots \geqslant e_r \geqslant 1$
  are integers. Then
  
  \smallskip
  (a) $\dim(\Lie(G_{\e})) = rp^{e_1 + \dots + e_r}$, and 
  
  \smallskip
  (b) $\dim(G_{\e}) = rp^{e_1+ \dots + e_r} - \sum_{i=1}^r  p^{s_i - i e_i}$.
\end{proposition} 

The remainder of this section will be devoted to proving Proposition~\ref{prop.dimensions}.
We will use the following notations.
\begin{enumerate}
    \item We fix the type $\e=(e_1,\ldots,e_r)$ and set $q_i = p^{e_i}$.

\item The infinitesimal group scheme $\alpha_{p^l}$  (over any commutative ring $S$) is defined as the kernel 
of the $l$-th power of the Frobenius map, in the exact sequence:
\begin{align*}
    0\to \alpha_{p^l} \to \mathbb{G}_a \xrightarrow[]{x\to x^{p^l}} \mathbb{G}_a \to 0.
\end{align*}
We will be particularly interested in the case, where $S = \Lambda_{\e}$.

\item  Suppose $X$ is a scheme over $\Lambda$, where $\Lambda$ is a finite-dimensional commutative $k$-algebra. We will denote by
$R_{\Lambda/k}(X)$ 
 the Weil restriction of the $\Lambda$-scheme $X$ to $k$ by $R_{\Lambda/k}(X)$. For generalities on Weil restriction, see Chapter 2 
and the Appendix in~\cite{milne}.
\end{enumerate}

Let us now consider the functor $\End(\Lambda_{\e})$ of algebra endomorphisms of $\Lambda_{\e}$ from the category of commutative $k$-algebras $\Comm_k$ (with $1$ but not necessarily finite-dimensional) to the category of sets $\Sets$. 

\begin{align*}
    \Comm_k &\xrightarrow{\End(\Lambda_{\e})} \Sets\\
    R &\xrightarrow{\End(\Lambda_{\e})(R)} \End_{R-alg}(\Lambda_{\e}\otimes_k R)
\end{align*} 
 
 \begin{lemma} \label{lem.endomorphisms}
 (a) The functor $\End(\Lambda_{\e})$ is represented by an irreducible (but non-reduced) affine $k$-scheme $X_{\e}$.
 
 \smallskip
 (b) $\dim(X_{\e}) = rp^{e_1+ \dots + e_r} -\sum_{i=1}^r  p^{s_i - i e_i}$. 
 
 \smallskip
 (c) $\dim(T_{\gamma}(X_{\e})) = rp^{e_1 + \dots + e_r}$ for any $k$-point $\gamma$ of $X_{\e}$. Here $T_{\gamma}(X_{\e})$ denotes the tangent space to 
 $X_{\e}$ at $\gamma$.
 \end{lemma}

 \begin{proof}
An endomorphism $F$ in $\End(\Lambda_{\e})(R)$ is uniquely determined by the images 
\[ F(x_1),F(x_2),\ldots,F(x_r) \in \Lambda_{\e}(R) \]
of the generators $x_1, \dots, x_r$ of $\Lambda_{\e}$.
These elements of $\Lambda_{\e}$ satisfy $F(x_i)^{q_i}=0$. Conversely, any $r$ elements $F_1,\ldots,F_r$ in $\Lambda_{\e}\otimes R$ satisfying $F_i^{q_i}=0$, give rise to an algebra endomorphism $F$ in $\End(\Lambda_{\e})(R)$.
  We thus have
 \begin{align*}
     \End(\Lambda_{\e})(R)&=\Hom_{R-alg}(\Lambda_{\e}\otimes_k R,\Lambda_{\e}\otimes R)\\
     &\cong\alpha_{q_1}(\Lambda_{\e}\otimes R)\times \ldots \times \alpha_{q_r}(\Lambda_{\e}\otimes R)\\
     &\cong R_{\Lambda_{\e}/k}(\alpha_{q_1})(R)\times \ldots \times R_{\Lambda_{\e}/k}(\alpha_{q_r})(R)\\
     &\cong \prod_{i=1}^r R_{\Lambda_{\e}/k}(\alpha_{q_i})(R)
 \end{align*}
We conclude that $\End(\Lambda_{\e})$ is represented by an affine $k$-scheme $X_{\e} = \prod_{i=1}^r R_{\Lambda_{\e}/k}(\alpha_{q_i})$.
(Note that $X_{\e}$ is isomorphic to $\prod_{i=1}^r R_{\Lambda_{\e}/k}(\alpha_{q_i})$ as a $k$-scheme only, not as a group scheme.)
To complete the proof of the lemma it remains to establish the following assertions:

\smallskip
For any $q_l\in \{ q_1,\ldots,q_r\}$ we have that

\smallskip
(a$'$) $R_{\Lambda_{\e}/k}(\alpha_{q_l})$ is irreducible,

(b$'$) $\dim\left(R_{\Lambda_{\e}/k}(\alpha_{q_l})\right) =p^{e_1+ \dots + e_r} - p^{s_l - l e_l}$ and

\smallskip
(c$'$) $\dim(T_{\gamma}(R_{\Lambda_{\e}/k}(\alpha_{q_l})))=p^{e_1+ \dots + e_r}$ for any $k$-point $\gamma$ of $R_{\Lambda_{\e}/k}(\alpha_{q_l})$.

\smallskip
To prove (a$'$), (b$'$) and (c$'$), we will write out explicit equations for $R_{\Lambda_{\e}/k}(\alpha_{q_l})$ in $R_{\Lambda_{\e}/k}(\mathbb{A}^1) \simeq \bbA_k(\Lambda_{\e})$.
We will work in the basis $\{x_1^{i_1}x_2^{i_2}\ldots x_r^{i_r} \}$ of monomials in $\Lambda_{\e}$, where $0\leqslant i_1<q_1$, $0\leqslant i_2 < q_2$, $\ldots$, $0\leqslant i_r <q_r$.
Over $\Lambda_{\e}$, $\alpha_{q_l}$ is cut out (scheme-theoretically) in $\mathbb{A}^1$ by the single equation $X^{q_l}=0$, where $X$ is a coordinate function on $\bbA^1$.
Since $x_i^{q_i} = 0$ for every $i$, writing
\begin{align*}
         X&=\sum Y_{i_1,\ldots,i_r} x_1^{i_1}x_2^{i_2}\ldots x_r^{i_r}
         \end{align*}
and expanding 
\begin{align*}
X^{q_l}  &=  \sum Y_{i_1,\ldots,i_r}^{q_l} x_1^{{q_l}i_1}x_2^{{q_l}i_2}\ldots x_r^{{q_l}i_r} \\
            \end{align*}
    we see that the only monomials appearing in the above sum are those for which 
    \[ q_li_1<q_1, \; \, q_li_2<q_2, \; \, \ldots, \; \, q_li_r<q_r . \]
Thus $R_{\Lambda_{\e}/k}(\alpha_{q_l})$ is cut out (again, scheme-theoretically) in $R_{\Lambda_{\e}/k}(\mathbb{A}^1) \simeq \bbA(\Lambda_{\e})$
by
\begin{align*}
        Y_{i_1,\ldots,i_{l-1},0,\ldots,0}^{q_l}=0 \text{ for  $0\leqslant i_1 <\frac{q_1}{q_l}$, $\ldots$, $0\leqslant i_{l-1}<\frac{q_{l-1}}{q_{l}}$,}  
            \end{align*}
where $Y_{i_1, \dots, i_r}$ are the coordinates in $\bbA(\Lambda_{\e})$. In other words,  $R_{\Lambda_{\e}/k}(\alpha_{q_l})$ is the
subscheme of $R_{\Lambda_{\e}/k}(\mathbb{A}^1) \simeq \bbA_k(\Lambda_{\e}) \simeq \bbA_k^{p^{e_1} + \dots + e_r}$ cut out
(again, scheme-theoretically) by $q_l$th powers of 
\[ \frac{q_1}{q_l}\frac{q_2}{q_l}\ldots \frac{q_{l-1}}{q_l} =
  p^{s_l - l e_l} \]
distinct coordinate functions. The reduced scheme $R_{\Lambda_{\e}/k}(\alpha_{q_l})_{\rm red}$ is thus isomorphic to
an affine space of dimension $p^{e_1+ \dots + e_r} -\sum_{i=1}^r  p^{s_i - i e_i}$. On the other hand, since $q_l$ is a power of $p$, the Jacobian criterion tells us that
the tangent space to $R_{\Lambda_{\e}/k}(\alpha_{q_l})$ at any $k$-point is the same as the tangent space 
to $\bbA(\Lambda_{\e}) = \bbA^{p^{e_1 + \dots + e_r}}$, and
(a$'$), (b$'$), (c$'$) follow.
      \end{proof}

\begin{proof}[Conclusion of the proof of Proposition~\ref{prop.dimensions}]
 The automorphism group scheme $G_{\e}$ is the group of invertible elements in $\End(\Lambda_{\e})$. In other words, the natural diagram
 \[ \xymatrix{  G_{\e} \ar@{->}[r] \ar@{->}[d] & \GL_N \ar@{->}[d] \\ 
                 \End(\Lambda_{\e}) \ar@{->}[r] &  \Mat_{N\times N} }
\]
where $N=\dim(\Lambda_{\e})=p^{e_1+\ldots+e_r}$, is Cartesian. Hence, $G_{\e}$ is an open subscheme of $X_{\e}$.
Since $X_{\e}$ is irreducible, Proposition~\ref{prop.dimensions} follows from Lemma~\ref{lem.endomorphisms}.
This completes the proof of Proposition~\ref{prop.dimensions} and thus of Theorem~\ref{thm.main'}.
%
%
 \end{proof}
 
 \section{Alternative proofs of Theorem~\ref{thm.main}}
 \label{sect.alternatives} 
 
The proof of the lower bound of Theorem~\ref{thm.main} given in Section~\ref{sect.lower-bound} section is the only one we know.
However, we have two other proofs for the upper bound (Proposition~\ref{prop.upper-bound}), in addition to the one given in Section~\ref{sect.upper-bound}. In this section we will briefly outline these arguments for the interested reader.

Our first alternative proof of Proposition~\ref{prop.upper-bound} is based on an explicit construction of the versal 
algebra $A_{ver}$ of type $(n, \e)$ whose existence is asserted by Proposition~\ref{prop.versal}. 
This construction is via generators and relations, by taking ``the most general" structure constants in~\eqref{e.structure}.
Versality of $A_{ver}$ constructed this way takes some work to prove; however,
once versality is established,  it is easy to see directly that $A_{ver}$ is a field and thus 
\[ \tau(n, \e) = \ed(A_{ver}) \leqslant \trdeg(K_{ver}/k) = n \sum_{i=1}^r  p^{s_i - i e_i} . \]

Our second alternative proof of Proposition~\ref{prop.upper-bound} is based on showing that the natural representation of
$G_{n, \e}$ on $V = \Lambda_{n, \e}^r$ is generically free. Intuitively speaking, this is clear: $\Lambda_{n, e}$ is generated by $r$ elements
as a $k$-algebra, so $r$-tuples of generators of $\Lambda_{n, \e}$ are dense in $V$ and
have trivial stabilizer in $G_{n, \e}$. The actual proof involves checking that the stabilizer 
in general position is trivial scheme-theoretically and not just on the level of points. Once generic freeness of this linear action 
is established, the upper bound of Proposition~\ref{prop.upper-bound} follows from the inequality
\[ \ed(G_{n, \e}) \leqslant \dim(V) - \dim(G_{n, \e}) \,  \]
see, e.g., Proposition 4.11 in~\cite{berhuy-favi}. To deduce the upper bound of Proposition~\ref{prop.upper-bound} from this inequality,
recall that 

\smallskip
$\tau(n, \e) = \ed(G_{n, \e})$ (see Corollary~\ref{cor.tau}(b)),

\smallskip
$\dim(V) = r \dim(\Lambda_{n, \e}) = nr \dim(\Lambda_{\bf e}) = nr p^{e_1 + \dots + e_r}$ (clear from the definition), and

\smallskip
$\dim(G_{n, \e}) = n \dim(G_{\e}) = nr p^{e_1+ \dots + e_r} -  n \sum_{i=1}^r  p^{s_i - i e_i}$ (see Proposition~\ref{prop.dimensions}(b)).

\section{The case, where $e_1 = \dots = e_r$}
\label{sect.ee}

In the special case, where $n = 1$ and $e_1 = \dots = e_r$, Theorem~\ref{thm.main} tells us that $\tau(n, \e) = r$.
In this section, we will give a short proof of the following stronger assertion (under the assumption that $k$ is perfect).

\begin{proposition} \label{prop.inseparable}
Let $\e = (e, \dots, e)$ ($r$ times) and $L/K$ be purely inseparable extension of type $\e$, with $k \subset K$.
Assume that the base field $k$ is perfect. Then $\ed_p(L/K) = \ed(L/K) = r$.
\end{proposition}

The assumption that $k$ is perfect is crucial here. Indeed, by Lemma~\ref{lem.existence}(b), there exists
a field extension $L/K$ of type $\e$. Setting $k = K$, we see that $\ed(L/K) = 0$, and the proposition fails.

The remainder of this section will be devoted to proving Proposition~\ref{prop.inseparable}.
We begin with two reductions. 

\smallskip
(1) It suffices to show that 
\begin{equation} \label{e.insep0}
\text{$\ed(L/K) = r$ for every field extension $L/K$ of type $\e$;}
\end{equation}
the identity $\ed_p(L/K)$ will then follow. Indeed,
$\ed_p(L/K)$ is defined as the minimal value of $\ed(L'/K')$ taken over all finite extensions $K'/K$ of degree prime to $p$.
Here $L' = L \otimes_K K'$. Since $[L:K]$ is a power of $p$, $L'$ is a field, so~\eqref{e.insep0} tells us that $\ed(L'/K') = r$.

\smallskip
(2) The proof of the upper bound,
\begin{equation} \label{e.insep1} \ed(L/K) \leqslant r  \end{equation}
is the same as in Section~\ref{sect.upper-bound}, but in this special case the argument is much simplified.
For the sake of completeness we reproduce it here.
Let $x_1, \dots, x_r$ be a normal generating sequence for $L/K$. 
By a theorem of Pickert (Theorem 1 in~\cite{rasala}), $x_1^{q}, \dots, x_r^q \in K$, where $q = p^e$. Set $a_i = x_i^q$ and $K_0 = k(a_1, \dots, a_r)$. The structure constants of $L$
relative to the $K$-basis $x_1^{d_1} \dots x_r^{d_r}$ of $L$, with $0 \leqslant d_1, \dots, d_r \leqslant q-1$ all lie in $K_0$.
Clearly $\trdeg(K_0/k) \leqslant r$; the inequality~\eqref{e.insep1} now follows from Lemma~\ref{lem.structure-constants}.

\smallskip
It remains to prove the lower bound, $\ed(L/K) \geqslant r$. Assume the contrary: $L/K$ descends to $L_0/K_0$ with
$\trdeg(K_0/k) < r$. By Lemma~\ref{lem.structure-constants}, $L_0/K_0$ further descends to $L_1/K_1$, where $K_1$ is finitely generated over $k$.
By Lemma~\ref{lem.type}, $L_1/K_1$ is a purely inseparable extension of type $\e$. After replacing $L/K$ by $L_1/K_1$, 
it remains to prove the following:

\begin{lemma} \label{lem.inseparable} Let $k$ be a perfect field and $K/k$ be a finitely generated field extension of transcendence degree $< r$.
There there does not exist a purely inseparable field extension $L/K$ of type $\e = (e_1, \dots, e_r)$, where $e_1 \geqslant \dots \geqslant e_r \geqslant 1$.
\end{lemma}

\begin{proof} Assume the contrary. Let $a_1, \dots, a_s$ be a transcendence basis for $K/k$. That is, $a_1, \dots, a_s$ are algebraically independent over $k$,
$K$ is algebraic and finitely generated (hence, finite) over $k(a_1, \dots, a_s)$ and $s \leqslant r-1$.
By Remark~\ref{rem.becker-maclane},
\begin{equation} \label{e.inseparable2} [L: L^p] \geqslant [L: (L^p \cdot K)] = p^r. \end{equation}
On the other hand, since $[L: k(a_1, \dots, a_s)] < \infty$, Theorem~3 in~\cite{becker-maclane}
tells us that
\begin{equation} \label{e.inseparable3}
[L: L^p] = [k(a_1, \dots, a_s): k(a_1, \dots, a_s)^p] = [k(a_1, \dots, a_s): k(a_1^p, \dots, a_s^p)] =  p^s < p^r . \end{equation}
Note that the second equality relies on our assumption that $k$ is perfect. The contradiction between~\eqref{e.inseparable2} 
and~\eqref{e.inseparable3} completes the proof of Lemma~\ref{lem.inseparable} and thus of Proposition~\ref{prop.inseparable}.
\end{proof}

\section*{Acknowledgements} We are grateful to Madhav Nori, Julia Pevtsova, Federico Scavia and Angelo Vistoli 
for stimulating discussions.

\bibliographystyle{amsalpha}


\end{document}